\RequirePackage{fix-cm}
\documentclass[envcountsect]{svjour3} 
\smartqed 

\usepackage{amsmath, mathrsfs, amscd, amsfonts, amssymb, graphicx, color} 
\usepackage[bookmarksnumbered, colorlinks, plainpages]{hyperref}
\hypersetup{colorlinks=true,linkcolor=blue, anchorcolor=blue, citecolor=blue, urlcolor=blue, filecolor=blue, pdftoolbar=true}

\numberwithin{equation}{section}

\usepackage{enumerate}
\usepackage{amsthm}
\usepackage{cite}
\usepackage{latexsym,verbatim}
\theoremstyle{plain}
\newtheorem{thm}[theorem]{Theorem}

\newtheorem{lem}[theorem]{Lemma}
\newtheorem{prop}[theorem]{Proposition}
\theoremstyle{definition}

\newtheorem{defn}[theorem]{Definition}
\newtheorem{ex}[theorem]{Example}
\theoremstyle{remark}

\newcommand{\R}{\mathbb{R}}
\newcommand{\C}{\mathbb{C}}

\newcommand{\T}{\mathbb{T}}
\newcommand{\D}{\mathbb{D}}
\newcommand{\N}{\mathbb{N}}

 \journalname{...}
\begin{document}

\title{Polynomial Approach to Cyclicity for Weighted $\ell^p_A$}
 
\author{Daniel Seco \and
 Roberto T\'ellez 
}


\institute{D. Seco \and R. T\'ellez
\at Instituto de Ciencias Matem\'aticas, Calle Nicol\'as Cabrera, 13-15, 28049 Madrid (Madrid), Spain
\and D. Seco \at
Universidad Carlos III de Madrid, Departamento de Matem\'aticas, Avenida de la Universidad 30, 28911 Legan\'es (Madrid), Spain\\
 \email{dseco@math.uc3m.es}
 \and
 R. T\'ellez \at
Universidad Aut\'onoma de Madrid, Departamento de Matem\'aticas, Calle Francisco Tom\'as y Valiente, 7, 28049 Madrid (Madrid), Spain\\
 \email{roberto.tellez@estudiante.uam.es} 
}

\date{Received: date / Accepted: date}

\maketitle

\begin{abstract}
In previous works, an approach to the study of cyclic functions in reproducing kernel Hilbert spaces has been presented, based on the study of so called \emph{optimal polynomial approximants}. In the present article, we extend such approach to the (non-Hilbert) case of spaces of analytic functions whose Taylor coefficients are in $\ell^p(\omega)$, for some weight $\omega$. When $\omega=\{(k+1)^\alpha\}_{k\in \N}$, for a fixed $\alpha \in \R$, we derive a characterization of the cyclicity of polynomial functions and, when $1<p<\infty$, we obtain sharp rates of convergence of the optimal norms.
\end{abstract}



\section{Introduction}
\label{intro}
A big part of complex analysis and operator theory over the unit disc $\D$ of the complex plane is devoted to the study of the shift operator $S$. The operator $S$ is defined on holomorphic functions  $f(z)$ by $Sf(z)=zf(z)$, therefore shifting the coefficients of the Taylor series around 0 to the next position. This operator acts boundedly on a large class of well-known spaces, as is the case for all of the following. Throughout the present article, $\N$ denotes the non-negative integers.
\begin{defn} \label{def1}
Let $1 \leq p < \infty$ and let $\omega =\{\omega_k\}_{k\in \N}$ be a sequence of positive real numbers  for which there exists a constant $C >0$ such that for all $t, k \in \N$ with $0 \leq t \leq k+1$ we have
\begin{equation}\label{eqn1}
C^{-1} \omega_k \leq \omega_{k+t} \leq C \omega_{k}.
\end{equation} Moreover, we assume that $\omega_0=1$ and
\begin{equation}\label{eqn2}
\lim_{k \rightarrow \infty} \frac{\omega_{k+1}}{\omega_k} = 1.
\end{equation}
We denote by $\ell^p_A(\omega)$ the space of analytic functions $f(z) =\sum_{k \in \N} a_k z^k$  over the disc $\D$ with finite norm
\begin{equation*}
\|f\|_{p,\omega}:= \left(\sum_{k=0}^{\infty} |a_k|^p\omega_k\right)^{1/p}.
\end{equation*} 
We also denote by $\ell^{\infty}_A(\omega)$ the space of analytic functions $f$ for which the norm 
\begin{equation*}
\|f\|_{\infty,\omega} = \sup_{n \in \mathbb N}|a_n|\omega_n 
\end{equation*}
is finite.
\end{defn}
When there is no possible confussion, we will sometimes denote the norm $\|\cdot \|_{p,\omega}$ just by $\|\cdot\|$.
The boundedness of $S$ is clear from either property \eqref{eqn1} or \eqref{eqn2}. From these properties it is also easy to infer the boundedness of $S^{-1}$, as well as the fact that the set $\mathcal{P}$ of all polynomials is a dense subspace of $\ell^p_A(\omega)$. Another by-product of these assumptions is that the disc is the biggest domain where all the elements of the space are holomorphic.

\begin{defn}\label{def2}
We say that a function $f \in \ell^p_A(\omega)$ is \emph{cyclic} if the smallest (closed) subspace of $\ell^p_A(\omega)$ that is invariant under the action of $S$ and contains $f$, $[f]$, is the whole $\ell^p_A(\omega)$.
\end{defn}
The study of cyclic functions in Banach spaces of analytic functions goes back a long time but it was systematically developed by Brown and Shields \cite{BS84}.
Trivially, the constant $1$ is always a cyclic function, and it plays a special role in the study of cyclicity: a function is indeed cyclic if and only if $1 \in [f]$, and that is equivalent to the existence of a sequence of polynomials $\{p_n\}_{n \in \N}$ such that $\|1-p_n f\|$ tends to 0 as $n$ tends to $\infty$. Each of these polynomials, $p_n$, can be taken in $\mathcal{P}_n$, the space of polynomials of degree less or equal to $n$, and this led \cite{BCLSS} to the introduction of \emph{optimal polynomial approximants}. This was generalized in \cite{FMS1}. There, only the case when $p=2$ is treated, as these are Hilbert spaces that are better understood. However, the definition makes sense in larger generality:

\begin{defn}\label{opa}
Let $f \in \ell_A^p(\omega)$, $n \in \N$ and $p_n \in \mathcal{P}_n$. We say that $p_n$ is an \emph{optimal polynomial approximant} to $1/f$ (in $\ell^p_A(\omega)$) if \[\|1-p_n f\| = \inf \|1-Pf\|,\] where the infimum is taken over all $P \in \mathcal{P}_n$. If $p_n$ is an optimal polynomial approximant of order $n$, we call $\|1-p_n f\|$ the \emph{optimal norm} of order $n$ (for $f$).
\end{defn} 

Notice that when $f$ is a function with at least one zero inside $\D$  in any of the spaces described, $f$ is not cyclic since $Pf$ will have that same zero independently of $P \in \mathcal{P}$. This disproves the cyclicity of $f$ because norm convergence of $1-Pf$ towards 0 implies pointwise convergence inside $\D$ of $Pf$ towards 1. On the other hand, if $f$ is a polynomial that does not have zeros in the \emph{closed} disc $\overline{\D}$, then $1/f$ has a Taylor series that converges beyond the boundary exponentially fast towards $1/f$, and so, if we choose $P_n$ to be the Taylor polynomial of $1/f$ of order $n \in \N$, we have $\|1-P_n f\|\rightarrow 0$ exponentially fast with $n$. Hence $f$ is cyclic in that case. Therefore, in the present text we study the simplest critical case: $f$ will be a polynomial whose zeros are contained in the unit circle. Our intention is to study the cyclicity of any such polynomial function $f$ and the decay (or not) with $n$ of the optimal norm for each space $\ell^p_A(\omega)$. Standard arguments extend the results to all functions with a convergence radius larger than 1.

We will focus on a special family of Banach spaces where the sequence $\omega$ is given by a parameter $\alpha$ in the sense that $\omega_k = (k+1)^\alpha$ for all $k \in \N$. The corresponding $\ell^p_A(\omega)$ space is denoted from here onwards by $\ell_A^{p,\alpha}$, and its norm $\|\cdot\|_{\ell^{p,\alpha}_A}$. This set of spaces naturally generalizes the Dirichlet-type spaces.

Going forward, we denote by $Z(f)$ the zero set of a polynomial $f$; for any function $g$ and any $k \in \N$, $\hat{g}(k)$ is the Taylor coefficient of $g$ around $0$ of order $k$; and $q$ will always be the H\"older conjugate of a fixed number $p$. We also need the following notation: For $z = r e^{i\theta}$, and $s \geq 0$ we denote \begin{equation}\label{spow}
z^{<s>}:=r^s e^{-i \theta}.
\end{equation} When $z=0$ we interpret that $z^{<s>}=0$.

With this in mind, our first result is as follows:
\begin{thm}\label{mainthm}
Let $1 < p < \infty$, $f \in \mathcal{P}$ of degree $d \in \N$, and $ Z(f) = \{z_1,...,z_m\} \subset \T$, with respective multiplicities $\{b_1,...,b_m\}$. Let $\omega$ be a weight as in Definition \ref{def1}, let $(p_n)_{n \in \N}$ be the sequence of optimal polynomial approximants to $1/f$ in the norm of $\ell^p_A(\omega)$ and $d_{t,n} = \left(\widehat{1-p_nf}(t)\right)^{<p-1>}\omega_t$. Then, for all $0 \leq t \leq n+d$ we have
\begin{equation}\label{eq:recu}
d_{t,n} = \sum_{i=1}^{m}\sum_{j=1}^{b_i} A_{i,j,n}t^{j-1}z_i^t,
\end{equation}
where the constants $A_{i,j,n}$ are the only solution to the following nonlinear system of $d$ equations: For $l = 1,\,...,\,m$, and $s= 1,...,b_l-1$,
\begin{equation}\label{eq:recucond}
\begin{aligned}
1 &= \sum_{t=0}^{n+d} \left( \sum_{i=1}^{m}\sum_{j=1}^{b_i} A_{i,j,n}t^{j-1}\frac{z_i^t}{\omega_t} \right)^{<q-1>} z_l^t
\\
0 &= \sum_{t=0}^{n+d} \left( \sum_{i=1}^{m}\sum_{j=1}^{b_i} A_{i,j,n}t^{j-1}\frac{z_i^t}{\omega_t} \right)^{<q-1>} t^s z_l^t.
\end{aligned}
\end{equation}
\end{thm}

When $p=1$, the result does not hold because the equations do not determine a unique function $1-p_n f$, but only their arguments. The lack of uniqueness is treated in the next Section.
Notice that when $t > n+d$, we automatically have $d_{t,n}=0$ and the description of $1-p_nf$ is then complete. The above result has a much more ellegant aspect when the zeros of $f$ are distinct, but we chose to present the unified formula since this was left as further work in \cite{BMS1}, at the same time that we extend the theory to non-Hilbert spaces. The simplified result is the following:

\begin{corollary}\label{maincor}
Let $f, p$  and $\omega$ be as in Theorem \ref{mainthm}, $f$ with simple zeros only. Then, for $0 \leq t \leq n+d$, we have 
\[d_{t,n} = \sum_{i=1}^d A_{i,n} z_i^t,\]
where $A_{i,n}$ are the only solution to the nonlinear system where $1 \leq l \leq d$: 
\begin{equation}\label{lastthing1}
1= \sum_{t=0}^{n+d} \left( \sum_{i=1}^d A_{i,n} \frac{z_i^t}{\omega_t} \right)^{<q-1>} z_{l}^t.\end{equation}
Moreover,
\begin{equation}\label{lastthing2}
\|1-p_nf\|^p_{p,\omega}=\sum_{i=1}^d A_{i,n}.
\end{equation}
\end{corollary}

In addition to these results, we will be able to compute up to a constant the exact rate of decay of the optimal norm in each case and determine which of the critical polynomials are cyclic in each space. As often in mathematical analysis, whenever we write $A(n) \approx B(n)$, we mean that there exists a constant $C>0$ independent of $n$ such that $C^{-1} B(n) \leq A(n) \leq CB(n)$.

\begin{thm}\label{thm2}
Let $f$ be a polynomial of degree $d$ such that $\emptyset \neq Z(f) \subset \T$, $1 < p<\infty$, and $\alpha \in \R$. Then $f$ is cyclic in $\ell^{p,\alpha}_A$ if and only if $\alpha \leq p-1$. In the case of $\ell^{1,\alpha}_A$, $f$ is cyclic if and only if $\alpha < 0$ while in $\ell_A^{\infty,\alpha}$, $f$ is cyclic if and only if $\alpha \leq 1$. Moreover, if $\{p_n\}_{n \in \N}$ is a sequence of optimal polynomial approximants to $1/f$ of the corresponding orders, for $1 \leq p < \infty$ we have
\begin{equation}\label{optnorm}
\|1-p_nf\|^p_{\ell^{p,\alpha}_A} \approx 
 \begin{cases} 
(n+d+1)^{\alpha +1-p}  & \quad if \quad  \alpha < p-1,
\\
\left(\log(n+d+2)\right)^{1-p}          &  \quad if \quad \alpha = p-1,
\\
1              &  \quad if \quad \alpha > p-1.
\end{cases}
\end{equation}
Finally, for $p = \infty$, we have
\begin{equation}\label{optnorminfty}
\|1-p_nf\|_{\ell^{\infty,\alpha}_A} \approx 
 \begin{cases} 
(n+d+1)^{\alpha - 1}  & \quad if \quad  \alpha < 1,
\\
\left(\log(n+d+2)\right)^{-1}          &  \quad if \quad \alpha = 1,
\\
1              &  \quad if \quad \alpha > 1.
\end{cases}
\end{equation} 
\end{thm}

The first part of this result is a generalization of a characterization of cyclicity for polynomials in Dirichlet-type spaces, achieved by Brown and Shields in \cite{BS84}. The estimates of the optimal norm generalize those in \cite{BCLSS}. In fact, the above result implies a characterization of cyclicity for functions that are holomorphic on a disc of radius bigger than 1 (following the ideas in \cite{FMS1}): for each $p$ and $\alpha$, such functions are cyclic if and only if they have no zeros inside the unit disc, and their zeros on the circle are the zeros of a cyclic polynomial. 

In the Hilbert space case, the existence and uniqueness of optimal polynomial approximants to $1/f$ for any $f$ not identically 0 (and for all $n \in \N$) follows directly from the existence and uniqueness of the orthogonal projection of $1$ onto the finite dimensional subspace of $[f]$ given by $\mathcal{P}_n \cdot f$, and in fact this idea gives an explicit method to find the optimal polynomial approximants. In the present article, we extend the current theory of such approximants to the general case of $1\leq p \leq \infty$, despite the failure of uniqueness in the extreme values of $p$. The Hilbert space proofs known to date for the results we will show rely heavily on the properties of orthogonal projection. Thus we will need a generalized concept of orthogonality due to Birkhoff and James, as well as properties of the metric projections on uniformly convex Banach spaces. The study of invariant subspaces of the shift in $\ell^p_A(1)$ has proved fruitful in \cite{Chengetal1, Chengetal2} and we will make use of many of the ideas in those articles. All these preliminaries, arithmetic properties related to the notation \eqref{spow}, general properties of the spaces in study, and the lack of uniqueness of optimal polynomial approximants  for $p=1$ and $p=\infty$ are introduced in Section \ref{Prelies} below. Then, in Section \ref{Sect3}, we prove Theorems \ref{mainthm} and \ref{thm2} as well as Corollary \ref{maincor}. Finally, we conclude in Section \ref{Sect5} with a simple example in which optimal polynomial approximants can be explicitly computed.

\section{Preliminaries} \label{Prelies}

\subsection{Metric projections and Birkhoff-James orthogonality}
For $1\leq p \leq \infty$, the spaces $\ell^p_A(\omega)$ are all Banach spaces when endowed with the norm $\|\cdot\|_{p,\omega}$. For $1<p<\infty$ the spaces $\ell^p_A(\omega)$ are, in fact, \emph{uniformly convex} Banach spaces \cite[pp.~95-96]{Convex}. This implies that for every vector $x$ and every closed subspace $V$ there is a unique \emph{metric projection} $\hat{x} \in V$ of $x$ onto $V$, meaning that $\|x-\hat{x}\|_B = \inf_{v \in V}\|x-v\|_B$. If $p=2$, they are Hilbert spaces and this metric projection coincides with the orthogonal projection.

Part of the difficulty that arises with developing a good approximation theory on $\ell^p_A(\omega)$ spaces comes precisely from the lack of a Hilbert structure. However, there exists a notion of orthogonality given by Birkhoff and James, used in detail in \cite{Chengetal1, Chengetal2}, which is valid in any Banach space $B$. We say that $x \in B$ is \emph{Birkhoff-James orthogonal} to $y \in B$, and we write $x \perp_B y$, if $\forall \alpha \in \C$ one has $\|x+ \alpha y\|_{B} \geq \|x\|_B$. When $B= \ell_A^p(\omega)$, we write $\perp_{p,\omega}$. If $B$ is a Hilbert space, this is equivalent to the usual definition of orthogonality. For an arbitrary Banach space, this relation is not linear (on the left-hand side parameter) and it is not symmetrical. However, one can sometimes reduce the concept of being Birkhoff-James orthogonal to a subspace to checking orthogonality to a basis, since in many cases this orthogonality is linear \emph{on the second term}. This is the case of the $\ell^p_A(\omega)$ spaces as we will see later.

When the value of $p$ is either $1$ or $\infty$, the spaces are no longer uniformly convex, and this implies that the metric projection may not be well defined, mainly due to the lack of uniqueness. We show now that this affects in particular the case of optimal polynomial approximants, which are non-unique whenever $p=1$ or $\infty$.

\begin{ex}\label{lemmabelow}
In $\ell^1_A(\omega)$, optimal polynomial approximants to \\$1/(1-\omega_1^{-1}z)$ of degree $0$ are not unique. In $\ell^{\infty}_A(\omega)$, optimal polynomial approximants to $1/(1-z^2)$ of degree 1 are not unique.
\end{ex}

\begin{proof}
Let $f(z) = 1-\omega_1^{-1}z$ and let $P_0(z) = c_0$ be a generic polynomial of degree $0$. Then \[\|1-P_0f\|_{1,\omega} = |1-c_0|+|c_0| \geq 1,\] which attains the equality for any $c_0 \in [0,1]$. That settles the case $p=1$. When $p=\infty$, let $g=1-z^2$ and let $P_1(z) = a +bz$ be a generic polynomial of degree $1$. Then \[\|1-P_1g\|_{\infty, \omega} = \sup \{|1-a|\omega_0, |b| \omega_1, |a|\omega_2, |b| \omega_3\}.\]
Denote $\omega^{-}=\inf\{\omega_0, \omega_2\}$, and $\omega^{+}=\sup\{\omega_1, \omega_3\}$. Any choice of $b$ such that \[|b| \leq \frac{\omega^{-}}{2\omega^{+}}\] gives the same result on the above norm expression, depending only on $a$. Any minimizing choice of $a$ makes $P_1$ optimal for values of $b$ within the given range.
 \end{proof}

\subsection{The spaces $\ell^{p,\alpha}_A$}

As mentioned in the introduction, we concentrate on the case where the sequence $\omega$ is given by a parameter $\alpha \in \R$, and $\omega_k = (k+1)^\alpha$ for all $k \in \N$, which we denoted by $\ell_A^{p,\alpha}$. The case $\alpha = 0$ constitutes the family of usual $\ell^p_A$  spaces, for which a theory of invariant subspaces has been furthered recently in \cite{Chengetal1, Chengetal2}. The cases $(p, \alpha) = (1,0), (2,-1), (2,0)$ and $(2,1)$ are respectively the Wiener algebra $\mathcal{A}(\T)$, the Bergman space $A^2$, the Hardy space $H^2$ and the Dirichlet space $\mathcal{D}$. These spaces are classical objects that have been studied in detail, especially in connection with invariant subspaces for the shift, and we refer the reader to the monographs \cite{Dur, DuS, EFKMR, Gar, HKZ} for more information. All these spaces satisfy condition \eqref{eqn1}:
First, suppose that $\alpha \geq 0$. Then for $0 \leq t \leq k+1$, we have
\[\omega_k \leq \omega_{k+t} = (k+t+1)^{\alpha} \leq (2k+2)^{\alpha} = 2^{\alpha}\omega_k.\]
If, on the contrary, $\alpha < 0$, then we have
\[2^{\alpha}\omega_k \leq \omega_{k+t} = (k+t+1)^{\alpha} \leq (k+1)^{\alpha} = \omega_k.\]
This shows \eqref{eqn1}. Condition \eqref{eqn2} is immediate. 
For any $\alpha$, the spaces $\ell^{2,\alpha}_A$ are examples of \emph{reproducing kernel Hilbert spaces} (RKHS), named Dirichlet-type spaces. The key property of a RKHS is the boundedness of the evaluation functionals at the points of the domain (in this case, $\D$). We prove this condition in general.

\begin{lem}\label{bddfunct}
Let $z_0 \in \D$, $1 \leq p \leq \infty$ and $\omega$ a sequence satisfying \eqref{eqn1} and \eqref{eqn2}. Then the functional assigning to a function $f \in \ell^p_A(\omega)$, the value $f(z_0)$ is bounded. In particular, norm convergence of a sequence of functions implies the pointwise convergence on all points of $\D$.
\end{lem}

\begin{proof}
Given a function $f \in \ell^{p}_A(\omega)$, write $f(z) = \sum_n a_nz^n$ and for $1\leq p < \infty$ and $n \in \N$ observe that $|a_n| \leq \|f\|_{p,\omega} \omega_n^{-1/p}$. In particular, for $z_0 \in \D$, $|f(z_0)| \leq \|f\|_{p,\omega} |h(z_0)|$, where $h(z) =  \sum_n  \omega_n^{-1/p}z^n$ is holomorphic on $\mathbb D$ by \eqref{eqn2} and the ratio test. If $p=\infty$ the same principle works, using $|a_n| \leq \|f\|_{p,\omega} \omega_n^{-1}$, instead. \end{proof}

Let us focus now on the relation between our spaces and Birkhoff-James orthogonality. The reason why such notion of orthogonality is appropriate for our problem is the following: given a closed subspace $V \subset B$ and a vector $x \in B$ with projection onto $V$ denoted by $\hat{x}$, for any $y \in V$ and $\alpha \in \mathbb C$ we have $\hat{x} - \alpha y \in V$, so \[\|(x-\hat{x}) + \alpha y\|_{B} = \|x - (\hat{x} - \alpha y)\|_B \geq \|x - \hat{x}\|_B\] since $\hat{x}$ is precisely the vector in $V$ that minimizes the distance to $x$. Thus, for all $y \in V$, we have \[x - \hat{x} \perp_B y,\]  in analogy with a well-known property of orthogonal projections on Hilbert spaces. 

Moreover, in the spaces $\ell^p_A(\omega)$ there is a nice characterization of Birkhoff-James orthogonality that resembles the Hilbert space situation. If $f,g \in \ell^p_A(\omega)$, then $f \perp_{p,\omega} g$ is equivalent to
\begin{equation}\label{eqn11}
 \sum_{n \in \N} |\hat{f}(n)|^{p-2}\overline{\hat{f}(n)}\hat{g}(n) \omega_n = 0.
\end{equation}

For $\omega \equiv 1$ (also called $\ell^p_A$ spaces), \eqref{eqn11} has been used in \cite{Chengetal1} and is easily generalized to $\ell^p_A(\omega)$ by observing that Birkhoff-James orthogonality is preserved by isometric isomorphisms of Banach spaces like

\begin{align*}
T: \: \ell^p_A(\omega)& \longrightarrow \ell^p_A
\\      \sum_{n \in \mathbb N} a_nz^n&         \longmapsto \sum_{n \in \mathbb N} a_n\omega_n^{1/p}z^n \:.
\end{align*}

This property seems essential to extending the work in \cite{BCLSS, BMS1, FMS1}.

In order to simplify these expressions we make use of the notation in \eqref{spow} from \cite{Chengetal1}. One can easily check the following properties.
\begin{lem}\label{prop3}
For $z, w \in \C \backslash \{0\}$, $s >0$, $1 < p <\infty$, $\alpha \in \R$ and $q = \frac{p}{p-1}$, we have
\begin{enumerate}
\item $(zw)^{<s>} = z^{<s>} w^{<s>}$,
\item $(z^{<s>})^{\alpha} = (z^{\alpha})^{<s>}$,
\item $z z^{<p-1>} = |z|^p$,
\item $(z^{<p-1>})^{<q-1>} = z$.
\end{enumerate}
\end{lem}

Notice that by Lemma \ref{prop3}, part (c), the operation $\cdot^{<p-1>}$ generalizes conjugation, in the sense that $z \overline{z} = |z|^2$. With this notation, the characterization of Birkhoff-James orthogonality in \eqref{eqn11} becomes
\begin{equation}\label{eqn12}
 \sum_{n \in \N} \hat{f}(n)^{<p-1>}\hat{g}(n) \omega_n = 0.
\end{equation}

\section{Proofs of main results} \label{Sect3}
\subsection{Proof of Theorem \ref{mainthm}}\label{ex2}

Recall that $\mathcal{P}_n$ denotes the space of polynomials of degree at most $n$ and, for a given $f \in \ell^p_A(\omega)$, the metric projection of $1$ onto $\mathcal{P}_n f$ is denoted directly by $p_n f$: for any $f$ not identically 0, the polynomial $p_n$ is uniquely determined by the projection of 1, and must be an optimal polynomial approximant to $1/f$. Now assume that $f \in \mathcal{P}_d$ is as in the statement, and with Taylor coefficients $a_k$ of order $k$, for $0 \leq k \leq d$. The fact that $p_n f$ is the projection of $1$ onto $\mathcal{P}_n f$, means exactly that for $j=0,...,n$ we have $1 - p_n f \perp_{p,\omega} z^jf$. Then from \eqref{eqn12} for $j=0,...,n$,  we obtain that
\begin{equation}\label{eqn21}
\sum_{t=0}^{n+d}d_{t,n} a_{t+j} = 0 .\end{equation}
The sum finishes in $t=n+d$ since for higher degrees all the terms $d_{t,n}$ are null. The relation between the coefficients in \eqref{eqn21} is a recurrence relation on $d_{t,n}$ whose general solution is of the form \eqref{eq:recu} as described in Section 2.1 of \cite{GK}. The conditions  \eqref{eq:recucond} are obtained by requiring that $(1 - p_n f)(z_l) = 1$ and that $(1 - p_n f)^{(s)}(z_l) = 0$ for any value of $s \geq 1$ lower than the multiplicity of the zero of $f$ at $z_l$. The existence of a solution to the system \eqref{eq:recucond} is clear since $1-p_nf$ exists. By construction, any solution $\{A_{i,j,n}\}_{i,j}$ to the system determines quantities $d_{t,n}$ which then define the Taylor coefficients of a polynomial of the form $1-Q f$ such that $1-Q f \bot_{p,\omega} z^jf$ for $j=0,...,n$. Since the metric projection is unique, so is $1-p_n f$ and thus $d_{t,n}$. This implies that there can only exist one solution $\{A_{i,j,n}\}_{i,j}$ to the system, since $\{t^{j-1}z_i^t\}_{i,j}$ is a basis for the space of solutions to the recurrence relation \eqref{eqn21} and the numbers $A_{i,j,n}$ are the coordinates of $d_{t,n}$ with respect to this basis.

\subsection{Proof of Theorem \ref{thm2}} 

Let $f$ be as in the statement of Theorem \ref{thm2}, $\omega=\{\omega_t\}_{t\in \N}=\{(t+1)^{\alpha}\}_{t \in \N}$ and consider the zero set of $f$ is formed by $z_1,...,z_m \in \T$, $m \leq d= \deg(f)$. Denote $$B_{t,n}=\widehat{1-p_nf}(t).$$ Since $f(z_l)=0$, then for $l=1,...,m$, we have that \[1=\sum_{t=0}^{n+d}  B_{t,n} z_l^t.\]
We can choose any linear combination of these equations to obtain
\[ \sum_{l=1}^m \lambda_l =\sum_{t=0}^{n+d}  B_{t,n} \sum_{l=1}^m \lambda_l z_l^t.\]
H\"older inequality yields
\[ \left| \sum_{l=1}^m \lambda_l \right| \leq \left(\sum_{t=0}^{n+d}  |B_{t,n}|^p \omega_t\right)^{1/p} \cdot \left(\sum_{t=0}^{n+d} \left| \sum_{l=1}^m \lambda_l z_l^t\right|^q \omega_t^{-q/p}\right)^{1/q}.\]
Calling $\mu_l:=\frac{\lambda_l}{  \sum_{l=1}^m \lambda_l }$, and from the definition of the norm in $\ell_A^p(\omega)$ we obtain
\[\|1-fp_n\|_{p,\omega} \geq \sup \left\{ \left(\sum_{t=0}^{n+d} \left| \sum_{l=1}^m \mu_l z_l^t\right|^q \omega_t^{-q/p}\right)^{-1/q} : \mu_l \in \C, \sum_{l=1}^m \mu_l=1\right\}.\]
Now we make the choice $\lambda_l=1$ (and hence $\mu_l=1/m$) for $l=1,...,m$. Bearing in mind that $Z(f) \subset \T$,
\[\left|\sum_{l=1}^m \mu_l z_l^t\right|  \leq 1.\]
This yields the simpler bound
\begin{equation}\label{eq:boundomega}
    \|1-fp_n\|_{p,\omega} \geq \left( \sum_{t=0}^{n+d} \omega_t^{-q/p}\right)^{-1/q}.
\end{equation}
Substituting $\omega_t = (t+1)^{\alpha}$ yields the lower estimate in Theorem \ref{thm2}.

For functions whose optimal polynomial approximants are hard to compute, we may just make an educated guess of other polynomials $P_n \in \mathcal{P}_n$ for whom the rate of convergence of $\|1-P_n f\|$ only differs from the optimal by a bounded multiplicative factor. We will first perform this choice for some simple functions and the general case will be derived later. For any $f$ for which we can find such a collection of polynomials, this would conclude the proof of Theorem \ref{thm2} (for that $f$). In what follows, consider fixed both the function $f$ in the statement of the Theorem and the space $\ell^{p,\alpha}_A$. We will only perform the computations for $1<p<\infty$, but the ideas apply directly to $p=1,\infty$ as well with the same arguments: the non-uniqueness does not affect the results and the only complication is in terms of expressions often containing a supremum instead of a sum of powers of some sequence. If $f$ is a polynomial of degree $d \geq 1$ with $Z(f) \cap \mathbb T \neq \emptyset$, we need $\alpha \leq p-1$ for $f$ to be cyclic (from the lower estimate already proved). We will now find some appropriate $\phi \in \mathbb R$ for which the optimal polynomial approximants to $1/f$ on $\ell^{2,\phi}_A$ are good enough on $\ell^{p,\alpha}_A$ (these will be our choice of $P_n$). Thanks to the Hilbert structure of $\ell^{2,\phi}_A$, optimal polynomial approximants are relatively easy to compute there and the estimates we need are already proved in \cite{BMS1}. 
We write
\begin{equation}\label{eqn61}
\delta_{k} := \left(\sum_{t=0}^{k} (t+1)^{-\alpha \cdot q/p}\right)^{1/q}, \quad \:\: k=0,...,n+d.
\end{equation}

\begin{prop}\label{th:2}
Let $f(z) = (z-e^{i\theta})^d$. Suppose $\alpha \leq p-1$, let $\phi =\frac{\alpha}{p-1}$ and denote the optimal polynomial approximants to $1/f$ on $\ell^{2,\phi}_A$ by $\{P_n\}_{n \in \N}$.
Then, 
\begin{equation*}
\|1-P_n f\|_{\ell^{p,\alpha}_A} = O(1/\delta_{n+d}).
\end{equation*}
\end{prop}
\begin{proof}
It is enough to consider the case where the zero of $f$ is $1$, as an appropriate rotation of the polynomials for $(z-1)^d$ will give the polynomials for any other $(z-e^{i\theta})^d$. From Theorem \ref{mainthm}, we have that \begin{equation*}
d_{t,n} :=  \widehat{1-P_nf}(t) \omega_t = \sum_{i=1}^d A_{i,n}t^{i-1},
\end{equation*} 
where the constants $A_{i,n}$ satisfy the linear system
\begin{align*}
1 &= \sum_{i=1}^d \left(\sum_{t=0}^{n+d}(t+1)^{-\phi}t^{i+1-2}\right) A_{i,n} 
\\
0 &= \sum_{i=1}^d \left(\sum_{t=0}^{n+d}(t+1)^{-\phi}t^{i+j-2}\right) A_{i,n},  \quad j=2,\ldots,d .
\end{align*}
This was already shown in Theorem 7.1 of \cite{BMS1} as well as the following estimates on the solution: with the notation there, let $E_{i,j,n}$ be the $(i,j)-$th element of the matrix $E_n$ defining this linear system where $A_{i,n}$ are the unknowns. If $\phi=i=j=1$, then \[E_{i,j,n} = \log(n+d+2)(1+o(1))\] and in any other case \[E_{i,j,n} = \frac{(n+d+1)^{i+j-1-\phi}}{i+j-1-\phi}(1 + o(1)).\] In \cite{BMS1}, the matrix was shown to be invertible. 

We can approximate the unknown values $A_{i,n}$ by making use of Cramer's rule. If $E_n^{(i,1)}$ is the matrix obtained by replacing the $i$th column of $E_n$ by the first vector of the canonical basis of $\mathbb R^d$, then $A_{i,n} = \det\, E_n^{(i,1)}/ \det \,E_n$. When $\phi = 1$, in the expansion of $\det \,E_n$ there are only terms of order $O(\log(n)n^{s})$ and $O(n^{s})$, $s=\frac{d \cdot (d-1)}{2}$, while $\det \, E_n^{(i,1)}$ has these same terms, but some of them multiplied by $0$ and some of them divided by $E_{1,i,n}$. If $i>1$, it is the leading terms that are multiplied by $0$ whereas if $i=1$ the opposite happens. In either case,
\begin{equation*}
(E_n^{-1})_{i,1} = \frac{C(1+o(1))}{\log(n+d+2)(n+d+1)^{i-1}}.
\end{equation*}

For $\phi<1$, we can even compute the exact constants by using a Cauchy matrix, $M$, given by $M_{i,j} = \frac{1}{i+j-1-\phi}$, for $i, j =1,...,d$. The inverse of such matrix is given in  \cite[pp.~512-515]{Matrices} and we obtain
\begin{equation*}
(E_n^{-1})_{i,j} = \frac{(M^{-1})_{i,j}(1+o(1))}{(n+d+1)^{i+j-1-\phi}}
\end{equation*}
by a similar argument as the one above. Since $A_{i,n} = (E_n^{-1})_{1,i}$ and $d_{t,n} = \sum_{i=1}^d A_{i,n}t^{i-1}$ for $t=0,...,n+d$, we see that
\begin{equation}\label{eqn53}
|d_{t,n}| \leq \sum_{i=1}^d |A_{i,n}|(n+d+1)^{i-1} = \begin{cases}
O\left(\frac{1}{\log(n+d+2)}\right)                     & \phi=1
\\
O((n+d+1)^{\phi - 1})                  & \phi<1
\end{cases}
\end{equation}

Notice these estimates do not depend on $t$.
We can finally recover $1-P_nf$ from all the values of $d_{t,n}$:
\[\|1-P_n f\|_{\ell^{p,\alpha}_A} = \left(\sum_{t=0}^{n+d}|d_{t,n}(t+1)^{-\phi}|^p(t+1)^{\alpha}\right)^{1/p}.\]

With the notation from \eqref{eqn61}, the right-hand side above is estimated making use of  \eqref{eqn53}, giving the bound $O(1/\delta_{n+d})$ and thus concluding the proof.
 \end{proof}

Recall now that $\mathcal{A(\T)}$ denotes the Wiener algebra, which is the space $\ell^1_A(1)$. We denote the Wiener norm by just $\|\cdot \|_1$. The Wiener norm plays a special role with regards to multiplication and therefore having uniform estimates will be useful to show our Theorem. Much of what we need now is already proved in \cite{BMS1}, Theorem 6.1. However, we also need estimates on the $(p, \alpha)$ norms.

\begin{prop}\label{th:1}
Let $f$ be as in the Theorem \ref{thm2} but with simple zeros only. For $\alpha \leq p-1$, let $\phi = \frac{\alpha}{p-1}$ and denote the optimal polynomial approximants to $1/f$ on $\ell^{2,\phi}_A$ by $\{P_n\}_{n \in \N}$.
Then, 
\begin{equation*}
\|1-P_n f\|_{\ell^{p,\alpha}_A} = O(1/\delta_{n+d})
\end{equation*} and there exists some constant $C$ independent of $n\in \N$ such that \[\|1-P_n f\|_{1} \leq C.\] 
\end{prop}

\begin{proof}
Write $d_{t,n} = \widehat{1-P_n f}(t) \cdot (t+1)^{\phi}$. From Theorem \ref{mainthm}, we have that $d_{t,n} = \sum_{i=1}^d A_{i,n} \overline{z_i}^t$, where the constants $A_{i,n}$ satisfy the linear system
\begin{equation*}
1 = \sum_{i=1}^d \left(\sum_{t=0}^{n+d}(t+1)^{-\phi}\overline{z_{i}}^tz_l^t\right) A_{i,n} \quad l=1,\ldots,d .
\end{equation*}
Thus, by Corollary 2.2 in \cite{BMS1}, we have 
\begin{equation*}
|A_{i,n}| = \begin{cases}
O((n+d+1)^{\phi-1})      & \phi < 1
\\
O(\log^{-1}(n+d+2))       & \phi = 1.
\end{cases}
\end{equation*}
The same bound, but with a different choice of constants, holds then for $d_{t,n} = \sum_{i=1}^d A_{i,n} \overline{z_i}^t$ yielding
\begin{equation}\label{eqn71}
\|1-P_n f\|^p_{\ell^{p,\alpha}_A} = \sum_{t=0}^{n+d}|d_{t,n}(t+1)^{-\phi}|^p(t+1)^{\alpha} = \sum_{t=0}^{n+d}|d_{t,n}|^p(t+1)^{-\phi}.
\end{equation}
Whenever $\alpha < p-1$ we obtain that 
\begin{equation*}
\|1-P_n f\|^p_{\ell^{p,\alpha}_A} = O((n+d+1)^{\alpha + 1 - p}).
\end{equation*}
For $ \alpha = p-1$, the right-hand side on \eqref{eqn71} is $O(\log^{1-p}(n+d+2))$ instead.
The Wiener norm estimates are contained in Theorem 6.1 of \cite{BMS1}.
 \end{proof}

What remains in order to prove the more general result is to notice that the $f$ in the statement of Theorem \ref{thm2} is divisible by $z-e^{i\theta}$ for some $\theta \in [0,2\pi)$ and divides \emph{some} polynomial of the form $g^{d_0}$ where $g$ is a polynomial with simple zeros only (take $Z(f)=Z(g)$ but the zeros in $g$ with multiplicity 1 and $d_0$ to be the maximum of the multiplicity of the zeros of $f$). In that case we know the Theorem \ref{thm2} for $g$ and the same upper estimate for the optimal norm for $g^{d_0}$ is automatically true for $f$ as well: $Q=g^{d_0}/f$ is a polynomial of fixed degree and the weight $\omega_t=(t+1)^{\alpha}$ is comparable to $\omega_{t+s}$ for a fixed $s \in \N$. From there, if $P_n$ are the polynomials achieving the estimate for $g^{d_0}$, $Q\cdot P_{n-d}$ achieve the corresponding estimate for $f$. The only remaining step is to control what happens to powers of functions for which we already have estimates of the optimal norm, for which we use the following Lemma, which we prove in Section \ref{Sect4}. We acknowledge that the proof is due to Raymond Cheng and we consider it of independent interest.

\begin{lem}\label{th:multiplier}
Let $f,g \in \ell^p_A(\omega)\cap\ell^1_A$ for some $1 \leq p \leq \infty$. Then, $fg \in \ell^p_A(\omega)\cap\ell^1_A$. Moreover, there exists a constant $C_{p,\omega}$ such that
\[ \|fg\|_{p,\omega} \leq C_{p,\omega}(\|f\|_{1}\|g\|_{p,\omega} + \|f\|_{p,\omega}\|g\|_{1}).\]
\end{lem}

With this tool, we are now ready to prove the only remaining question to establish Theorem \ref{thm2}.

\begin{prop}\label{th:main}
Let $f= \prod_{i=1}^m(z-z_i)^{d_i}$ be a polynomial of degree $d$ with $Z(f) \subset \mathbb T$, and $g(z)=\prod_{i=1}^m(z-z_i)$. For $\alpha \leq p-1$, let $\phi = \frac{\alpha}{p-1}$ and denote the optimal polynomial approximants to $1/g$ on $\ell^{2,\phi}_A$ by $\{q_n\}_{n \in \N}$. Let  $d_0 = \max_{1\leq i\leq m} d_i$, denote $\sigma(n) = \left \lfloor{\frac{n+d}{d_0}}\right \rfloor-m$ and write $P_{n} = (q_{\sigma(n)}g)^{d_0}/f$. 
Then \begin{equation*}
\|1-P_n f\|_{\ell^{p,\alpha}_A} = O(1/\delta_{n+d}).
\end{equation*}
\end{prop}
\begin{proof}
We claim first that $\|1-(q_{\sigma(n)}g)^{d_0}\|_1$ is uniformly bounded as $n \rightarrow \infty$, which we will prove by induction on $d_0$.  For $d_0=1$, that is part of Proposition \ref{th:1}. Otherwise, notice that $\ell^1_A$ is a multiplicative algebra and therefore,
\begin{equation}\label{eqn81}
\|1- (q_{\sigma(n)}g)^{d_0+1}\|_{1} \lesssim \|1  - (q_{\sigma(n)}g)^{d_0}\|_1 +  \|q_{\sigma(n)}g\|_1^{d_0}\|1 - q_{\sigma(n)}g\|_{1}.
\end{equation}
Since the 3 quantities on the right-hand side of \eqref{eqn81} are uniformly bounded in $\ell^1_A$-norm, this concludes our first claim.
To prove our result, we proceed again by induction on $d_0$. Suppose we have proved our result for $d_0 \leq k$ and take $d_0=k+1$. Then $\|1-P_n f\|_{\ell^{p,\alpha}_A}$ is bounded by
\begin{equation}\label{eqn82}
\|1  - (q_{\sigma(n)}g)^{k}\|_{\ell^{p,\alpha}_A}  + \|1  - (q_{\sigma(n)}g)\|_{\ell^{p,\alpha}_A}
+ \|(1-(q_{\sigma(n)}g)^{k}) \cdot (1- q_{\sigma(n)}g)\|_{\ell^{p,\alpha}_A}.\end{equation}
The first two terms in \eqref{eqn82} would be decaying at the speed claimed in the Proposition, by the induction hypothesis. The third term is controlled by applying Lemma \ref{th:multiplier} together with our previous claim. \end{proof}

\subsection{Estimate for multiplication} \label{Sect4}

In order to establish Theorem \ref{thm2}, we still need to give the proof of Lemma \ref{th:multiplier}. As we mentioned above, the proof is due to Raymond Cheng. We consider the result natural enough for it to exist in the literature but we have not found it anywhere.
\begin{proof}
Since $\ell^1_A$ is a multiplicative algebra, we only need to check that $fg \in \ell^p_A(\omega)$. If $k \notin 2\mathbb N$, $\frac{k}{2}$ will denote $\left \lfloor{\frac{k}{2}}\right \rfloor+1 $.
Suppose first that $1 \leq p < \infty$ and consider the convex function $\varphi(x) = |x|^p$. It follows from Jensen's inequality that

\begin{equation}\label{eq:jensen}
\left(\sum_{t=0}^{k/2} |a_t||b_{k-t}| \right)^p \leq \left(\sum_{t=0}^{k/2} |a_t|\right)^{p-1}\left(\sum_{t=0}^{k/2} |a_t||b_{k-t}|^p  \right).
\end{equation}

For $f = \sum_k a_kz^k$ and $g = \sum_k b_kz^k$, the product $fg$ satisfies
\begin{equation*}
\|fg\|_{p,\omega}^p = \sum_{k=0}^{\infty} |\sum_{t=0}^k a_t b_{k-t}|^p\omega_k,
\end{equation*}
which can be splitted into the values of $t \leq k/2$ and the rest. We obtain
\begin{equation}\label{eqn92}
\|fg\|_{p,\omega}^p \leq \sum_{k=0}^{\infty} \left(\sum_{t=0}^{k/2} |a_t| |b_{k-t}|\right)^p\omega_k + \sum_{k=0}^{\infty} \left(\sum_{t=0}^{k/2} |b_t| |a_{k-t}|\right)^p\omega_k.
\end{equation}
We denote the two terms on the right-hand side of \eqref{eqn92} by $A_1(f,g)$ and $A_1(g,f)$. It is clear that it suffices to show that $A_1(f,g)$ is bounded (and then apply that to $(g,f)$ too). Indeed, making use of \eqref{eq:jensen}, we see that
\begin{equation}\label{eqn93}
A_1(f,g) \leq \|f\|_1^{p-1} \sum_{s=0}^{\infty}\sum_{t=0}^{s}|a_t||b_s|^p\omega_{s+t}.
\end{equation}
Using the doubling property of the weight, the term $\omega_{s+t}$ in \eqref{eqn93} can be substituted by a constant $C_\omega$ times $\omega_s$. What we have then is that
\begin{equation*}
A_1(f,g) \leq C_{\omega}\|f\|^p_1\|g\|^p_{p,\omega}.
\end{equation*}
For $p = \infty$, the proof is similar: the $\|\cdot \|_{\infty,\omega}$ norm is separated in the cases when $t \leq k/2$ or not. This yields
\begin{equation*}
\|fg\|_{\infty,\omega} \leq \sup_{k \in \mathbb N} \left( \sum_{t=0}^{k/2}|a_t||b_{k-t}|\omega_k   \right) + \sup_{k \in \mathbb N} \left( \sum_{t=0}^{k/2}|b_t||a_{k-t}|\omega_k   \right),
\end{equation*}
which we can again denote by $A_2(f,g) + A_2(g,f)$. To find a bound for $A_2(f,g)$, the same method as before works. \end{proof}

\subsection{Proof of Corollary \ref{maincor}}

The first part of the Corollary \ref{maincor} is just a particular case of Theorem \ref{mainthm}, when the zeros are simple. What remains is to show \eqref{lastthing2}.
Multiplying both sides of \eqref{lastthing1} by $A_{l,n}$ and summing on $l=1,...,d$, we obtain
\[\sum_{i=1}^d A_{i,n} = \sum_{t=0}^{n+d} \left| \sum_{i=1}^d A_{i,n}z_i^t\right|^q \omega_{t}^{1-q},  \]
where we have used Lemma \ref{prop3} part (c).
Now notice that 
\[ \left| \sum_{i=1}^d A_{i,n}z_i^t\right|^q \omega_{t}^{1-q} = \left|\widehat{1-p_nf}(t)\right|^p \omega_t,\]
and then we can sum on both sides for $t=0,...,n+d$, concluding the proof.

\section{A simple example} \label{Sect5}
We present here an explicit computation of optimal polynomial approximants to $1/f$, where $f(z) = 1 - z^d$ for $d \in \mathbb N$, $d \geq 1$, on any $\ell^p_A(\omega)$. Again, we perform the computations for $1<p<\infty$ but the same ideas apply in the extremal cases.  First, we reduce the problem to studying the case $d=1$. Write $P_n = \sum_{t=0}^n c_tz^t$ for any polynomial $P_n \in \mathcal{P}_n$ and notice that
\[\|1-fP_n\|^p_{p,\omega} = |1-c_0|^p + \sum_{t=1}^{d-1} |c_t|^p\omega_t + \sum_{t=d}^{n} |c_t-c_{t-d}|^p\omega_t + \sum_{t=n+1}^{n+d}|c_{t-d}|^p\omega_t.\]
Hence, the choice $c_t = 0$ for $t \notin d\N$ cannot increase the norm $\|1-fP_n\|_{p,\omega}$, which means that optimal polynomial approximants are of the form $P_n(z) = Q(z^d)$ for some polynomial $Q$ of degree $\left \lfloor{\frac{n}{d}}\right \rfloor$. The same computation shows that polynomials of such form satisfy $\|1-fP_n\|^p_{p,\omega} = \|1-(1-z)Q\|_{p,\tilde{\omega}}^p$, where $\tilde{\omega}_t = \omega_{dt}$, implying that $P_n$ is the optimal polynomial approximant to $1/(1-z^d)$ of order $n$ on $\ell^p_A(\omega)$ precisely when $Q$ is the optimal polynomial approximant to $1/(1-z)$ of order $\left \lfloor{\frac{n}{d}}\right \rfloor$ on $\ell^p_A(\tilde{\omega})$. Thus, we may restrict ourselves to the case $d=1$.
\\
\\
Now from \eqref{eq:boundomega} we see that, in order to obtain optimal polynomial approximants to $1/(1-z)$ of order $n$ on $\ell^p_A(\omega)$, it is enough to find $p_n \in \mathcal{P}_n$ such that $\|1-(1-z)p_n\|^p_{p,\omega} = \left( \sum_{t=0}^{n+1} \omega_t^{-q/p}\right)^{-p/q}$. As in \eqref{eqn61}, we write
$$
\delta_{k} := \left(\sum_{t=0}^{k} \omega_t^{-q/p}\right)^{1/q}, \quad \:\: k=0,...,n+1.$$
Then we let
$$p_n(z) := \sum_{t=0}^{n} (1-\delta^q_t/\delta^q_{n+1})z^t $$
and using $-q/p = -q+1$ it is immediate to check that $\|1-(1-z)p_n\|^p_{p,\omega} = \sum_{t=0}^{n+1} \omega_t^{-q/p}\delta_{n+1}^{-pq} = \delta_{n+1}^{-p}$, as desired. One way of obtaining this solution is by solving system \eqref{lastthing1}, which is a direct computation in this case.

\begin{acknowledgements}
We acknowledge financial support from the Spanish Ministry of Economy and
Competitiveness, through the ``Severo Ochoa Programme for Centers of
Excellence in R\&D'' (SEV-2015-0554) and through grant
MTM2016-77710-P. We are also grateful to Raymond Cheng and to an anonymous referee for helpful comments and careful reading.
\end{acknowledgements}

\bibliographystyle{amsplain}

\begin{thebibliography}{}


\bibitem{BCLSS} \textsc{B\'en\'eteau, C., Condori, A., Liaw,
C., Seco, D.,} and \textsc{Sola, A.}, Cyclicity in Dirichlet-type
spaces and extremal polynomials,  \emph{J. Anal. Math.} {\bf 126}
(2015) 259--286.

\bibitem{BMS1} \textsc{B\'en\'eteau, C., Manolaki, M.,} and \textsc{Seco, D.}, 
Boundary behavior of optimal polynomial approximants, \emph{Constr. Approx.} (2020), online first, \url{https://doi.org/10.1007/s00365-020-09508-z}.

\bibitem{Convex} \textsc{Brezis, H.}
\emph{Functional Analysis, Sobolev Spaces and Partial Differential Equations}, Springer-Verlag, New York, 2011.

\bibitem{BS84} \textsc{Brown, L.} and \textsc{Shields, A.},  Cyclic vectors in the Dirichlet space,
\emph{Trans. Amer. Math. Soc.} {\bf 285} (1984), 269-304.

\bibitem{Chengetal1} \textsc{Cheng, R., Mashreghi, J.}, and \textsc{Ross, W. T.}, Inner functions and zero sets for $\ell_A^p$, \emph{Trans. Amer. Math. Soc.} {\bf 372} (2019), Issue 3, 2045-2072.

\bibitem{Chengetal2} \textsc{Cheng, R., Mashreghi, J.}, and \textsc{Ross, W. T.}, Inner functions in reproducing kernel spaces, \emph{Analysis of operators on function spaces}, Birkh\"auser, 2019.

\bibitem{Dur} \textsc{Duren, P. L.},
\emph{Theory of $H^p$ spaces}, Academic Press, New York, 1970.

\bibitem{DuS} \textsc{Duren, P. L.} and \textsc{Schuster, A.},
\emph{Bergman spaces}, AMS, Providence, RI, 2004.

\bibitem{EFKMR} \textsc{El-Fallah, O., Kellay, K., Mashreghi, J.,} and \textsc{Ransford, T.}, 
\emph{A primer on the Dirichlet space}, Cambridge Tracts in Math. {\bf 203}, Cambridge University Press, 2014.

\bibitem{FMS1} \textsc{Fricain, E., Mashreghi, J.}, and
\textsc{Seco, D.}, Cyclicity in Reproducing Kernel Hilbert Spaces of
analytic functions, \emph{Comput. Methods Funct. Theory} (2014)
Issue 14, 665-680.

\bibitem{Gar} \textsc{Garnett, J. B.},
\emph{Bounded analytic functions}, Academic Press Inc., 1981.

\bibitem{GK} \textsc{Greene, D. H.} and \textsc{Knuth, D. E.},
\emph{Mathematics for the Analysis of Algorithms}, Modern Birkh\"auser Classics, 3$^{rd}$ edition, 2008.

\bibitem{HKZ} \textsc{Hedenmalm, H., Korenblum, B.,} and \textsc{Zhu, K.},
\emph{Theory of Bergman spaces}, Springer, New York, 2000.

\bibitem{Matrices} \textsc{Higham, N. J.},
\emph{Accuracy and Stability of Numerical Algorithms}, 2$^{nd}$ edition, SIAM, 2002.


\end{thebibliography}

\end{document}